\documentclass[11pt]{amsart}
\usepackage[margin=30mm]{geometry}
\usepackage{amsmath,amssymb}
\usepackage{amsthm}

\newtheorem{thm}{Theorem}[section]

\newtheorem{lem}{Lemma}[section]

\newtheorem{defi}{Definition}[section]
\newtheorem{ex}{Example}[section]
\newtheorem{rem}{Remark}[section]

\begin{document}

\title{Contractive shadowing property of dynamical systems}
\author{Noriaki Kawaguchi}
\address{Research Institute of Science and Technology, Tokai University, 4-1-1 Kitakaname, Hiratsuka, Kanagawa 259-1292, Japan}
\email{gknoriaki@gmail.com}

\begin{abstract}
For topological dynamical systems defined by continuous self-maps of compact metric spaces, we consider the contractive shadowing property, i.e., the Lipschitz shadowing property such that the Lipschitz constant is less than $1$. We prove some basic properties of contractive shadowing and show that non-degenerate homeomorphisms do not have the contractive shadowing property. Then, we consider the case of ultrametric spaces. We also discuss the spectral decomposition of the chain recurrent set under the assumption of contractive shadowing. Several examples are given to illustrate the results. 
\end{abstract}

\maketitle

\markboth{NORIAKI KAWAGUCHI}{Contractive shadowing property of dynamical systems}

\section{Introduction}

Throughout, $X$ denotes a compact metric space endowed with a metric $d$. Let $f\colon X\to X$ be a continuous map and let $\xi=(x_i)_{i\ge0}$ be a sequence of points in $X$. For $\delta>0$, $\xi$ is called a {\em $\delta$-pseudo orbit} of $f$ if $d(f(x_i),x_{i+1})\le\delta$ for all $i\ge0$. For $\epsilon>0$, $\xi$ is said to be {\em $\epsilon$-shadowed} by $x\in X$ if $d(f^i(x),x_i)\leq \epsilon$ for all $i\ge 0$. For $L>0$, we say that $f$ has the {\em $L$-Lipschitz shadowing property} if there is $\delta_0>0$ such that for any $0<\delta\le\delta_0$, every $\delta$-pseudo orbit of $f$ is $L\delta$-shadowed by some point of $X$. If $f$ has the $L$-Lipschitz shadowing property for some $L>0$, then $f$ is said to have the {\em Lipschitz shadowing property}.

\begin{rem}
\normalfont
\begin{enumerate}
\item We say that $f$ has the (standard) {\em shadowing property} if for any $\epsilon>0$, there is $\delta>0$ such that every $\delta$-pseudo orbit of $f$ is $\epsilon$-shadowed by some point of $X$.
\item For $0<\alpha\le1$, let $d^\alpha$ be a metric on $X$ defined by
\[
d^\alpha(x,y)=d(x,y)^\alpha
\]
for all $x,y\in X$. For any continuous map $f\colon X\to X$, we easily see that if $f$ has the $L$-Lipschitz shadowing property with respect to $d$, then $f$ has the $L^\alpha$-Lipschitz shadowing property with respect to $d^\alpha$.
\end{enumerate}
\end{rem}

We know that for any closed Riemannian manifold $M$, a $C^1$-diffeomorphism $f\colon M\to M$ is structurally stable if and only if $f$ has the Lipschitz shadowing property \cite{PS}. In \cite{LS,S2,S3}, it is shown that a positively expansive map or an expansive homeomorphism $f\colon X\to X$ has the shadowing property if and only if $f$ has the Lipschitz shadowing property with respect to an equivalent metric $D$ to $d$. Moreover, it is questioned in \cite{S1} whether {\em every} homeomorphism $f\colon X\to X$ with the shadowing property satisfies the Lipschitz shadowing property with respect to an equivalent metric $D$ to $d$. In this paper, we consider the $L$-Lipschitz shadowing such that the constant $L$ is less than 1.

\begin{defi}
\normalfont
We say that a continuous map $f\colon X\to X$ has the {\em contractive shadowing property} if $f$ has the $L$-Lipschitz shadowing property for some $0<L<1$.
\end{defi}

We first relate the contractive shadowing to the so-called {\em h-shadowing}. We recall the definition of h-shadowing from \cite{BGO} (see also \cite{BGOR}). Given a continuous map $f\colon X\to X$ and $\delta>0$, a finite sequence $(x_i)_{i=0}^{k}$ of points in $X$, where $k>0$ is a positive integer, is called a {\em $\delta$-chain} of $f$ if $d(f(x_i),x_{i+1})\le\delta$ for every $0\le i\le k-1$.  We say that $f$ has the {\em h-shadowing property} if for any $\epsilon>0$, there is $\delta>0$ such that for every $\delta$-chain $(x_i)_{i=0}^k$ of $f$, there is $x\in X$ such that $d(f^i(x),x_i)\le\epsilon$ for all $0\le i\le k-1$ and $f^k(x)=x_k$. Then, the following theorem holds. 

\begin{thm}
If a continuous map $f\colon X\to X$ has the contractive shadowing property, then $f$ satisfies the h-shadowing property.
\end{thm}

\begin{rem}
\normalfont
\begin{enumerate}
\item It is easy to see that every continuous map $f\colon X\to X$ with the h-shadowing property is an open map, that is, for any open subset $U$ of $X$, $f(U)$ is an open subset of $X$.
\item A continuous map $f\colon X\to X$ is said to be {\em transitive} (resp.\:{\em mixing}) if for any non-empty open subsets $U,V$ of $X$, it holds that $f^j(U)\cap V\ne\emptyset$ for some $j>0$ (resp.\:for all $j\ge i$ for some $i>0$). We say that $f$ is {\em locally eventually onto} if for any non-empty open subset $U$ of $X$, there is $i>0$ such that $f^i(U)=X$. We easily see that if $f$ is locally eventually onto, then $f$ is mixing, and the converse holds when $f$ has the h-shadowing property.
\item Let $f\colon X\to X$ be a continuous map and let $\xi=(x_i)_{i\ge0}$ be a sequence of points in $X$. For $\delta>0$, $\xi$ is called a {\em $\delta$-limit-pseudo orbit} of $f$ if $d(f(x_i),x_{i+1})\le\delta$ for all $i\ge0$, and
\[
\lim_{i\to\infty}d(f(x_i),x_{i+1})=0.
\]
For $\epsilon>0$, $\xi$ is said to be {\em $\epsilon$-limit shadowed} by $x\in X$ if $d(f^i(x),x_i)\leq \epsilon$ for all $i\ge 0$, and
\[
\lim_{i\to\infty}d(f^i(x),x_i)=0.
\]
We say that $f$ has the {\em s-limit shadowing property} if for any $\epsilon>0$, there is $\delta>0$ such that every $\delta$-limit-pseudo orbit of $f$ is $\epsilon$-limit shadowed by some point of $X$. By Theorem 3.7 of \cite{BGO}, we know that if $f$ has the h-shadowing property, then $f$ has the s-limit shadowing property.  
\end{enumerate}
\end{rem}

Here, we recall a basic lemma and prove it for completeness. For $x\in X$ and $r>0$, $B_r(x)$ denotes the closed $r$-ball centered at $x$:
\[
B_r(x)=\{y\in X\colon d(x,y)\le r\}.
\]  

\begin{lem}
Let $f\colon X\to X$ be a continuous map. For any $\delta,\epsilon>0$, if
\[
B_{\delta+\epsilon}(f(x))\subset f(B_{\epsilon}(x)) 
\]  
for all $x\in X$, then
\begin{itemize}
\item for every $\delta$-chain $(x_i)_{i=0}^k$ of $f$, there is $x\in X$ such that $d(f^i(x),x_i)\le\epsilon$ for all $0\le i\le k-1$ and $f^k(x)=x_k$,
\item every $\delta$-pseudo orbit of $f$ is $\epsilon$-shadowed by some point of $X$.
\end{itemize}
\end{lem}

\begin{proof}
For any $x,y,z\in X$ with $d(f(x),y)\le\delta$ and $d(y,z)\le\epsilon$, since
\[
d(f(x),z)\le d(f(x),y)+d(y,z)\le\delta+\epsilon,
\]
we have
\[
z\in B_{\delta+\epsilon}(f(x))\subset f(B_\epsilon(x)). 
\]
It follows that
\[
B_\epsilon(y)\subset f(B_\epsilon(x))
\]
for all $x,y\in X$ with $d(f(x),y)\le\delta$. Let $(x_i)_{i=0}^k$ be a $\delta$-chain of $f$. For any $0\le i\le k-1$, we have
\[
B_\epsilon(x_{i+1})\subset f(B_\epsilon(x_i))
\]
because $d(f(x_i),x_{i+1})\le\delta$. It follows that there is a sequence $p_i\in B_\epsilon(x_i)$, $0\le i\le k$, with $p_k=x_k$ and $f(p_i)=p_{i+1}$ for all $0\le i\le k-1$, implying
\[
p_0\in\bigcap_{i=0}^{k-1} f^{-i}(B_\epsilon(x_i))
\]
and $f^k(p_0)=p_k=x_k$. It also follows that for any $\delta$-pseudo orbit of $\xi=(x_i)_{i\ge0}$ of $f$,
\[
\bigcap_{i=0}^\infty f^{-i}(B_\epsilon(x_i))=\bigcap_{k>0}\bigcap_{i=0}^{k-1} f^{-i}(B_\epsilon(x_i))\ne\emptyset;
\]
therefore, $\xi$ is $\epsilon$-shadowed by some $q\in X$. This completes the proof of the lemma. 
\end{proof}

Following the terminology of \cite{BGO}, we say that a continuous map $f\colon X\to X$ is {\em ball expanding} if there is $0<L<1$ and $\delta_0>0$ such that
\[
B_\delta(f(x))\subset f(B_{L\delta}(x)) 
\]  
for all $0<\delta\le\delta_0$ and $x\in X$.

\begin{rem}
\normalfont
\begin{enumerate}
\item In \cite{BGO}, it is shown that for any continuous map $f\colon X\to X$, if $f$ is ball expanding, then $f$ has the h-shadowing property (see Theorem 4.3 of \cite{BGO}). In fact, this is a consequence of Lemma 1.1. The proof of Theorem 1.1 in Section 2 is by a direct method and does not use Lemma 1.1. Given a continuous map $f\colon X\to X$, $0<L<1$, and $\delta_0>0$, assume that for any $0<\delta\le\delta_0$, every $\delta$-pseudo orbit of $f$ is $L\delta$-shadowed by some point of $X$. As a consequence of the proof of Theorem 1.1 in Section 2, we obtain
\[
B_\delta(f(x))\subset f(B_{\frac{L}{1-L}\delta}(x)) 
\]
for all $x\in X$ and $0<\delta\le\delta_0$, thus if $L/(1-L)<1$, or equivalently, $L<1/2$, then $f$ is ball expanding.

\item Let $f\colon X\to X$ be a ball expanding map and let $0<L<1$, $\delta_0>0$ be as in the definition of ball expanding. For any $0<\Delta\le(1-L)\delta_0$, since
\[
0<\frac{1}{1-L}\Delta\le\delta_0,
\]
we have
\[
B_{\Delta+\frac{L}{1-L}\Delta}(f(x))=B_{\frac{1}{1-L}\Delta}(f(x))\subset f(B_{\frac{L}{L-1}\Delta}(x))
\]
for all $x\in X$. By Lemma 1.1, we see that every $\Delta$-pseudo orbit of $f$ is $L\Delta/(1-L)$-shadowed by some point of $X$. It follows that if $L/(1-L)<1$, or equivalently, $L<1/2$, then $f$ has the contractive shadowing property. 
\end{enumerate}
\end{rem}

For a continuous map $f\colon X\to X$ and $\delta>0$, a $\delta$-chain $(x_i)_{i=0}^k$ of $f$ is called a {\em $\delta$-cycle} of $f$ if $x_0=x_k$. We say that $x\in X$ is a {\em chain recurrent point} for $f$ if for every $\delta>0$, there is a $\delta$-cycle $(x_i)_{i=0}^{k}$ of $f$ with $x_0=x_k=x$. We denote by $CR(f)$ the set of chain recurrent points for $f$. Let $Per(f)$ denote the periodic points for $f$:
\[
Per(f)=\bigcup_{i>0}\{x\in X\colon f^i(x)=x\}.
\]
By Theorem 3.4.2 of \cite{AH}, we know that if a continuous surjection $f\colon X\to X$ is c-expansive and has the shadowing property, then $f$ satisfies $CR(f)=\overline{Per(f)}$. The next theorem shows that the same folds for any continuous map $f\colon X\to X$ with the $L$-Lipschitz shadowing property when $0<L<1/2$.

\begin{thm}
If a continuous map $f\colon X\to X$ has the $L$-Lipschitz shadowing property for some $0<L<1/2$, then $f$ satisfies $CR(f)=\overline{Per(f)}$.
\end{thm}

For $M\ge1$, we say that a map $f\colon X\to X$ is {\em $M$-Lipschitz continuous} if there is $\delta_0>0$ such that 
\[
d(f(x),f(y))\le Md(x,y)
\]
for all $x,y\in X$ with $d(x,y)\le\delta_0$. The next theorem shows that $M^{-1}$ is a lower bound of $L>0$ such that an $M$-Lipschitz continuous map  $f\colon X\to X$ satisfying $X=CR(f)$ and the $L$-Lipschitz shadowing property is not trivial. Note that if $X$ is a finite set, then any map $f\colon X\to X$ is continuous and satisfies the $L$-Lipschitz shadowing property for all $L>0$.

\begin{thm}
Given $M\ge1$, if an $M$-Lipschitz continuous map $f\colon X\to X$ satisfies $X=CR(f)$ and has the $L$-Lipschitz shadowing property for some $0<L<M^{-1}$, then $X$ is a finite set.
\end{thm}

We say that a homeomorphism $f\colon X\to X$ is {\em equicontinuous} if for any $\epsilon>0$, there is $\delta>0$ such that $d(x,y)\le\delta$ implies
\[
\sup_{i\in\mathbb{Z}}d(f^i(x),f^i(y))\le\epsilon
\]
for all $x,y\in X$. By Theorem 6.1 of \cite{BGO}, we know that a homeomorphism $f\colon X\to X$ has the h-shadowing property if and only if $f$ is equicontinuous and $X$ is totally disconnected; therefore, for example, an odometer satisfies the h-shadowing property. In contrast, the next theorem shows that a homeomorphism $f\colon X\to X$ does not have the contractive shadowing property except for the degenerate case.

\begin{thm}
If a homeomorphism $f\colon X\to X$ has the contractive shadowing property, then $X$ is a finite set.
\end{thm}

We consider the case where $d$ is an ultrametric. A metric $d$ on $X$ is called an {\em ultrametric} on $X$ if
\[
d(x,z)\le\max\{d(x,y),d(y,z)\}
\]
for all $x,y,z\in X$. We know that a compact metric space $X$ endowed with a metric $d$ is totally disconnected if and only if there is an ultrametic $D$ on $X$ equivalent to $d$.

When $d$ is an ultrametric on $X$, we can improve the constant $1/2$ in Theorem 1.2 to $1$, as the following theorem shows. 

\begin{thm}
Let $d$ be an ultrametric on $X$. If a continuous map $f\colon X\to X$ has the contractive shadowing property, then $f$ satisfies $CR(f)=\overline{Per(f)}$.
\end{thm}

The following theorem shows that when $d$ is an ultrametric on $X$ and $0<L<1$, the $L$-Lipschitz shadowing property is equivalent to ball expanding for the same 
$L$.

\begin{thm}
Let $d$ be an ultrametric on $X$. For $0<L<1$ and a continuous map $f\colon X\to X$, the following conditions are equivalent
\begin{enumerate}
\item $f$ has the $L$-Lipschitz shadowing property,
\item there is $\delta_0>0$ such that
\[
B_\delta(f(x))\subset f(B_{L\delta}(x)) 
\]  
for all $0<\delta\le\delta_0$ and $x\in X$.
\end{enumerate}
\end{thm}

We discuss the spectral decomposition of the chain recurrent set. From Theorem 3.4.4 of \cite{AH}, we know that if a continuous surjection $f\colon X\to X$ is c-expansive and has the shadowing property, then $CR(f)$ admits a decomposition into finitely many {\em basic sets} (due to Smale) and every basic set admits a cyclic decomposition into finitely many {\em elementary sets} (due to Bowen). We shall show that an analogous result holds for any continuous map $f\colon X\to X$ with the contractive shadowing property.

As the following theorem shows, for $L>0$, the $L$-Lipschitz shadowing is not lost by restricting the map to its chain recurrent set.

\begin{thm}
For $L>0$, if a continuous map $f\colon X\to X$ has the $L$-Lipschitz shadowing property, then
\[
f|_{CR(f)}\colon CR(f)\to CR(f)
\]
also satisfies the $L$-Lipschitz shadowing property.
\end{thm}

The next two theorems describe the decomposition of $CR(f)$ for any continuous map $f\colon X\to X$ with the contractive shadowing property (precise definitions and notations are given in Section 4). 

\begin{thm}
For $0<L<1$, if a continuous map $f\colon X\to X$ has the $L$-Lipschitz shadowing property, then $\mathcal{C}(f)$ is a finite set and for all $C\in\mathcal{C}(f)$,
\[
f|_{C}\colon C\to C
\]
satisfies the $L$-Lipschitz shadowing property.
\end{thm}

\begin{thm}
For $0<L<1$, if a continuous map $f\colon X\to X$ is chain transitive and has the $L$-Lipschitz shadowing property, then $\mathcal{D}(f)$ is a finite set and for all $D\in\mathcal{D}(f)$,
\[
f^{|\mathcal{D}(f)|}|_{D}\colon D\to D
\]
satisfies the $L$-Lipschitz shadowing property.
\end{thm}

\begin{rem}
\normalfont
\begin{enumerate}
\item Let $X=\{0,1\}^\mathbb{N}$ and let $\sigma\colon X\to X$ be the shift map. Since $\sigma$ is positively expansive and has the shadowing property, by Proposition 4.1 of \cite{BGO}, $\sigma$ has the h-shadowing property. Let $Y$ be a Cantor set in $[0,1]$ and let $id_Y\colon Y\to Y$ be the identity map. Since $id_Y$ is an equicontinuous homeomorphism and $Y$ is totally disconnected, by Proposition 6.1 of \cite{BGO}, $id_Y$ has the h-shadowing property. It follows that
\[
\sigma\times id_Y\colon X\times Y\to X\times Y
\]
satisfies the h-shadowing property, but we have
\[
\mathcal{C}(\sigma\times id_Y)=\{X\times\{y\}\colon y\in Y\},
\]
which is an uncountable set.
\item Let $g\colon X_m\to X_m$ be an odometer. Since $g$ is an equicontinuous homeomorphism and $X_m$ is totally disconnected, by Proposition 6.1 of \cite{BGO}, $g_m$ has the h-shadowing property, but we have
\[
\mathcal{D}(g)=\{\{x\}\colon x\in X_m\},
\]
which is an uncountable set.
\end{enumerate}
\end{rem}

This paper consists of five sections. In Section 2, we prove Theorems 1.1, 1.2, 1.3, and 1.4. In Section 3, we prove Theorems 1.5 and 1.6. In Section 4,  we prove Theorems 1.7, 1.8, and 1.9; and by Theorems 1.8 and 1.9, we provide an alternative proof of Theorem 1.4. Finally, in Section 5, we give several examples. 

\section{Proofs of Theorems 1.1, 1.2, 1.3 and 1.4}

In this section, we prove Theorems 1.1, 1.2, 1.3, and 1.4. First, we prove Theorem 1.1.

\begin{proof}[Proof of Theorem 1.1]
Since $f$ has the contractive shadowing property, there are $0<L<1$ and $\delta_0>0$ such that for any $0<\delta\le\delta_0$, every $\delta$-pseudo orbit of $f$ is $L\delta$-shadowed by some point of $X$. Fix $0<\delta\le\delta_0$ and a $\delta$-chain $(x_i)_{i=0}^k$ of $f$. We have $y_0\in X$ with
\[
d(f^i(y_0),x_i)\le L\delta
\]
for all $0\le i\le k$. Since $d(f^k(y_0),x_k)\le L\delta$,
\[
(y_0,f(y_0),\dots,f^{k-1}(y_0),x_k)
\]
is an $L\delta$-chain of $f$, so we have $y_1\in X$ with
\[
d(f^i(y_1),f^i(y_0))\le L^2\delta
\]
for all $0\le i\le k-1$ and $d(f^k(y_1),x_k)\le L^2\delta$. Since
\[
(y_1,f(y_1),\dots,f^{k-1}(y_1),x_k)
\]
is an $L^2\delta$-chain of $f$, we have $y_2\in X$ with
\[
d(f^i(y_2),f^i(y_1))\le L^3\delta
\]
for all $0\le i\le k-1$ and $d(f^k(y_2),x_k)\le L^3\delta$. Proceeding inductively, we obtain a sequence $y_n\in X$, $n\ge0$, satisfying for every $n\ge0$, $d(f^k(y_n),x_k)\le L^{n+1}\delta$ and
\[
d(f^i(y_{n+1}),f^i(y_n))\le L^{n+2}\delta
\]
for all $0\le i\le k-1$. Because
\[
d(y_0,x_0)+\sum_{n=0}^\infty d(y_{n+1},y_n)\le L\delta+\sum_{n=0}^\infty L^{n+2}\delta=\frac{L}{1-L}\delta<\infty,
\]
$(y_n)_{n\ge0}$ is a Cauchy sequence and so $\lim_{n\to\infty}y_n=x$ for some $x\in X$. Since
\[
d(f^i(y_N),x_i)\le d(f^i(y_0),x_i)+\sum_{n=0}^{N-1}d(f^i(y_{n+1}),f^i(y_n))\le L\delta+\sum_{n=0}^\infty L^{n+2}\delta=\frac{L}{1-L}\delta
\]
for all $N\ge1$ and $0\le i\le k-1$, letting $N\to\infty$, we obtain
\[
d(f^i(x),x_i)\le\frac{L}{1-L}\delta
\]
for all $0\le i\le k-1$. Moreover, by $d(f^k(y_n),x_k)\le L^{n+1}\delta$ for all $n\ge0$, letting $n\to\infty$, we obtain $f^k(x)=x_k$. This shows that $f$ satisfies the h-shadowing property, completing the proof.
\end{proof}

\begin{rem}
\normalfont
Given a continuous map $f\colon X\to X$, $0<L<1$, and $\delta_0>0$, assume that for any $0<\delta\le\delta_0$, every $\delta$-pseudo orbit of $f$ is $L\delta$-shadowed by some point of $X$. Let $0<\delta\le\delta_0$. Then, for any $x,y\in X$ with $d(f(x),y)\le\delta$, since $(x,y)$ is a $\delta$-chain of $f$, as shown in the above proof of Theorem 1.1, we have $f(z)=y$ for some $z\in X$ with
\[
d(z,x)\le\frac{L}{1-L}\delta.
\]
Thus, we obtain
\[
B_\delta(f(x))\subset f(B_{\frac{L}{1-L}\delta}(x)) 
\]
for all $x\in X$ and $0<\delta\le\delta_0$.
\end{rem}

Next, we prove Theorem 1.2.

\begin{proof}[Proof of Theorem 1.2]
By assumption, we have $0<L<1/2$ and $\delta_0>0$ such that for any $0<\delta\le\delta_0$, every $\delta$-pseudo orbit of $f$ is $L\delta$-shadowed by some point of $X$. Fix $0<\delta\le\delta_0$ and a $\delta$-cycle $(x_i)_{i=0}^k$ of $f$. We have $y_0\in X$ with
\[
d(f^i(y_0),x_i)\le L\delta
\]
for all $0\le i\le k$. From $x_0=x_k$, it follows that
\[
d(f^k(y_0),y_0)\le d(f^k(y_0),x_k)+d(y_0,x_0)\le L\delta+L\delta=2L\delta.
\]
Since 
\[
(y_0,f^i(y_0),\dots,f^{k-1}(y_0),y_0)
\]
is a $2L\delta$-cycle of $f$, we have
\[
\max\{d(y_1,y_0),d(f^k(y_1),y_0)\}\le L(2L)\delta
\]
for some $y_1\in X$. It follows that
\[
d(f^k(y_1),y_1)\le d(f^k(y_1),y_0)+d(y_1,y_0)\le L(2L)\delta+L(2L)\delta=(2L)^2\delta.
\]
Since
\[
(y_1,f^i(y_1),\dots,f^{k-1}(y_1),y_1)
\]
is a $(2L)^2\delta$-cycle of $f$, we have
\[
\max\{d(y_2,y_1),d(f^k(y_2),y_1)\}\le L(2L)^2\delta
\]
for some $y_2\in X$.  It follows that
\[
d(f^k(y_2),y_2)\le d(f^k(y_2),y_1)+d(y_2,y_1)\le L(2L)^2\delta+L(2L)^2\delta=(2L)^3\delta.
\]
Proceeding inductively, we obtain a sequence $y_n\in X$, $n\ge0$, satisfying $d(f^k(y_n),y_n)\le(2L)^{n+1}\delta$ and
\[
\max\{d(y_{n+1},y_n),d(f^k(y_{n+1}),y_n)\}\le L(2L)^{n+1}\delta
\]
for every $n\ge0$. Because
\[
d(y_0,x_0)+\sum_{n=0}^\infty d(y_{n+1},y_n)\le L\delta+\sum_{n=0}^\infty L(2L)^{n+1}\delta=\frac{L}{1-2L}\delta<\infty,
\]
$(y_n)_{n\ge0}$ is a Cauchy sequence and $\lim_{n\to\infty}y_n=x$ for some $x\in X$ with \[
d(x,x_0)\le\frac{L}{1-2L}\delta.
\]
Moreover, by $d(f^k(y_n),y_n)\le(2L)^{n+1}\delta$ for every $n\ge0$, letting $n\to\infty$, we obtain $f^k(x)=x$. This shows that $f$ satisfies $CR(f)=\overline{Per(f)}$, completing the proof.
\end{proof}

We give a proof of Theorem 1.3.

\begin{proof}[Proof of Theorem 1.3]
By assumption, we have $0<L<M^{-1}$ and $\delta_0>0$ such that the following conditions are satisfied
\begin{itemize}
\item $d(a,b)\le\delta_0$ implies $d(f(a),f(b))\le Md(a,b)$ for all $a,b\in X$,
\item for any $0<\delta\le\delta_0$, every $\delta$-pseudo orbit of $f$ is $L\delta$-shadowed by some point of $X$.
\end{itemize}
Fix $0<\delta\le\delta_0$ and a $\delta$-chain $(x_i)_{i=0}^k$ of $f$. Let us show that $f^k(x_0)=x_k$. We have $y_0\in X$ with
\[
d(f^i(y_0),x_i)\le L\delta
\]
for all $0\le i\le k$. It follows that $d(f(x_0),f(y_0))\le ML\delta$ and $d(f^k(y_0),x_k)\le L\delta$. Since 
\[
(x_0,f(y_0),\dots,f^{k-1}(y_0),x_k)
\]
is an $ML\delta$-chain of $f$, we have
\[
\max\{d(y_1,x_0),d(f^k(y_1),x_k)\}\le L(ML)\delta
\]
for some $y_1\in X$. It follows that $d(f(x_0),f(y_1))\le(ML)^2\delta$. Since
\[
(x_0,f(y_1),\dots,f^{k-1}(y_1),x_k)
\]
is an $(ML)^2\delta$-chain of $f$, we have
\[
\max\{d(y_2,x_0),d(f^k(y_2),x_k)\}\le L(ML)^2\delta
\]
for some $y_2\in X$. It follows that $d(f(x_0),f(y_2))\le(ML)^3\delta$. Proceeding inductively, we obtain a sequence $y_n\in X$, $n\ge0$, satisfying
\[
\max\{d(y_n,x_0),d(f^k(y_n),x_k)\}\le L(ML)^n\delta
\]
for every $n\ge0$. Since $0<ML<1$, letting $n\to\infty$, we obtain $f^k(x_0)=x_k$, proving the claim. Since $X=CR(f)$, for any $p\in X$, there is a $\delta_0/2$-cycle $(z_i)_{i=0}^m$ of $f$ with $z_0=z_m=p$. For any $q\in X$ with $d(p,q)\le\delta_0/2$, since
\[
(z_0,z_1,\dots,z_{m-1},q)
\]
is a $\delta_0$-chain of $f$, we have $f^m(p)=f^m(z_0)=q$. This implies that $f^m(p)=p$ and $p$ is an isolated point of $f$. Since $p\in X$ is arbitrary, we conclude that $X$ consists of isolated points and so is a finite set, completing the proof.
\end{proof}

We need a lemma for the proof of Theorem 1.4. For a continuous map $f\colon X\to X$ and $x\in X$, the {\em $\omega$-limit set} $\omega(x,f)$ of $x$ for $f$ is defined as the set of $y\in X$ such that $\lim_{j\to\infty}f^{i_j}(x)=y$ for some sequence $0\le i_1<i_2<\cdots$. We define a subset $R(f)$ of $X$ by 
\[
R(f)=\{x\in X\colon x\in\omega(x,f)\}.
\]
By Lemma 25 in Chapter IV of \cite{BC}, we know that $R(f)=R(f^k)$ for all $k>0$.

\begin{lem}
If a continuous map $f\colon X\to X$ satisfies $X=R(f)$ and has the contractive shadowing property, then $X$ is a finite set.
\end{lem}

\begin{proof}
Since $f$ has the contractive shadowing property, there are $0<L<1$ and $\delta_0>0$ such that for any $0<\delta\le\delta_0$, every $\delta$-pseudo orbit of $f$ is $L\delta$-shadowed by some point of $X$. In order to obtain a contradiction, we assume that
\[
0<d(p,q)<\delta_0
\]
for some $p,q\in X$. Let $\delta=d(p,q)$ and $L<K<1$. By induction on $n$, we shall show that there is a sequence $r_n\in X$, $n\ge0$, such that
\[
\max\{d(p,r_n),d(q,r_n)\}\le K^n\delta
\]
for all $n\ge0$, which implies $p=q$, a contradiction. First, for $n=0$, letting $r_0=p$, we have
\[
\max\{d(p,r_0),d(q,r_0)\}=d(p,q)=\delta=K^0\delta.
\]
Assume that
\[
\max\{d(p,r_n),d(q,r_n)\}\le K^n\delta
\]
for some $n\ge0$ and $r_n\in X$. We take $\epsilon>0$ so small that $\epsilon+K^n\delta\le\delta_0$ and $L(\epsilon+K^n\delta)<K^{n+1}\delta$. Since $X=R(f)=CR(f)$, we have three $\epsilon$-cycles $(x_i)_{i=0}^k$, $(y_i)_{i=0}^k$, $(z_i)_{i=0}^k$ of $f$ with $x_0=x_k=p$, $y_0=y_k=r_n$, and $z_0=z_k=q$. By the induction assumption, we see that
\[
\xi=(x_0,x_1,\dots,x_{k-1},y_0,y_1,\dots,y_{k-1},z_0,z_1,\dots,z_{k-1},z_0,z_1,\dots,z_{k-1},z_0,z_1,\dots,z_{k-1},\dots)
\]
is an $(\epsilon+K^n\delta)$-pseudo orbit of $f$ and so $L(\epsilon+K^n\delta)$-shadowed by some $r_{n+1}\in X$. Note that
\[
d(p,r_{n+1})=d(r_{n+1},x_0)\le L(\epsilon+K^n\delta)<K^{n+1}\delta.
\]
Also, by the choice of $\xi$, we have
\[ 
d(q,f^{km}(r_{n+1}))=d(f^{km}(r_{n+1}),z_0)\le L(\epsilon+K^n\delta)<K^{n+1}\delta
\]
for all $m\ge2$. By $r_{n+1}\in R(f)=R(f^k)$, we obtain
\[
d(q,r_{n+1})\le d(q,f^{km}(r_{n+1}))+d(r_{n+1},f^{km}(r_{n+1}))<K^{n+1}\delta
\]
for some $m\ge2$, thus the induction is complete. We conclude that $X$ consists of isolated points and so is a finite set, proving the lemma.
\end{proof}

By Lemma 2.1, we prove Theorem 1.4.
 
\begin{proof}[Proof of Theorem 1.4]
Since $f$ has the contractive shadowing property, Theorem 1.1 implies that $f$ satisfies the h-shadowing property. For any $x\in X$ and $\delta>0$, since
\[
\lim_{i\to\infty}d(f^i(x),\omega(x,f))=0,
\]
there is a $\delta$-chain $(x_i)_{i=0}^k$ of $f$ with $x_0=x$ and $x_k\in\omega(x,f)$. Note that we have
\[
f^{-1}(\omega(x,f))=\omega(x,f)
\]
because $f$ is a homeomorphism. From  the h-shadowing property of $f$, it follows that
\[
x\in\overline{\bigcup_{k>0}f^{-k}(\omega(x,f))}=\omega(x,f).
\]
Since $x\in X$ is arbitrary, we obtain $X=R(f)$, thus by Lemma 2.1, we conclude that $X$ is a finite set, proving the theorem.
\end{proof}

\section{Ultrametric case: proofs of Theorems 1.5 and 1.6}
 
In this section, we prove Theorems 1.5 and 1.6. First, we prove Theorem 1.5. The proof is similar to that of Theorem 1.2 but we use the assumption that $d$ is an ultrametric. 

\begin{proof}[Proof of Theorem 1.5]
Since $f$ has the contractive shadowing property, there are $0<L<1$ and $\delta_0>0$ such that for any $0<\delta\le\delta_0$, every $\delta$-pseudo orbit of $f$ is $L\delta$-shadowed by some point of $X$. Fix $0<\delta\le\delta_0$ and a $\delta$-cycle $(x_i)_{i=0}^k$ of $f$. We have $y_0\in X$ with
\[
d(f^i(y_0),x_i)\le L\delta
\]
for all $0\le i\le k$. From $x_0=x_k$, it follows that
\[
d(f^k(y_0),y_0)\le\max\{d(f^k(y_0),x_k),d(y_0,x_0)\}\le L\delta.
\]
Since 
\[
(y_0,f(y_0),\dots,f^{k-1}(y_0),y_0)
\]
is an $L\delta$-cycle of $f$, we have
\[
\max\{d(y_1,y_0),d(f^k(y_1),y_0)\}\le L^2\delta
\]
for some $y_1\in X$. It follows that
\[
d(f^k(y_1),y_1)\le\max\{d(f^k(y_1),y_0),d(y_1,y_0)\}\le L^2\delta.
\]
Since
\[
(y_1,f(y_1),\dots,f^{k-1}(y_1),y_1)
\]
is an $L^2\delta$-cycle of $f$, we have
\[
\max\{d(y_2,y_1),d(f^k(y_2),y_1)\}\le L^3\delta
\]
for some $y_2\in X$.  It follows that
\[
d(f^k(y_2),y_2)\le\max\{d(f^k(y_2),y_1),d(y_2,y_1)\}\le L^3\delta.
\]
Proceeding inductively, we obtain a sequence $y_n\in X$, $n\ge0$, satisfying $d(f^k(y_n),y_n)\le L^{n+1}\delta$ and
\[
\max\{d(y_{n+1},y_n),d(f^k(y_{n+1}),y_n)\}\le L^{n+2}\delta
\]
for every $n\ge0$. It follows that $(y_n)_{n\ge0}$ is a Cauchy sequence and $\lim_{n\to\infty}y_n=x$ for some $x\in X$ with
\[
d(x,x_0)\le L\delta.
\]
Moreover, by $d(f^k(y_n),y_n)\le L^{n+1}\delta$ for all $n\ge0$, letting $n\to\infty$, we obtain $f^k(x)=x$. This shows that $f$ satisfies $CR(f)=\overline{Per(f)}$, completing the proof.
\end{proof}

Then, we prove Theorem 1.6.

\begin{proof}[Proof of Theorem 1.6]
First, we prove the implication $(1)\implies(2)$. Since $f$ has the $L$-Lipschitz shadowing property, there is $\delta_0>0$ such that for any $0<\delta\le\delta_0$, every $\delta$-pseudo orbit of $f$ is $L\delta$-shadowed by some point of $X$. Fix $0<\delta\le\delta_0$ and $y\in B_\delta(f(x))$. Since $(x,y)$ is a $\delta$-chain of $f$, we have
\[
\max\{d(z_0,x),d(f(z_0),y)\}\le L\delta
\]
for some $z_0\in X$. Since $(z_0,y)$ is a $L\delta$-chain of $f$, we have
\[
\max\{d(z_1,z_0),d(f(z_1),y)\}\le L^2\delta
\]
for some $z_1\in X$. Since $(z_1,y)$ is a $L^2\delta$-chain of $f$, we have
\[
\max\{d(z_2,z_1),d(f(z_2),y)\}\le L^3\delta
\]
for some $z_2\in X$. Proceeding inductively, we obtain a sequence $z_n\in X$, $n\ge0$, satisfying
\[
\max\{d(z_{n+1},z_n),d(f(z_{n+1}),y)\}\le L^{n+2}\delta
\]
for every $n\ge0$. It follows that $(z_n)_{n\ge0}$ is a Cauchy sequence and $\lim_{n\to\infty}z_n=z$ for some $z\in X$ with \[
d(z,x)\le L\delta.
\]
Moreover, by $d(f(z_{n+1}),y)\le L^{n+2}\delta$ for all $n\ge0$, letting $n\to\infty$, we obtain $f(z)=y$. Since $y\in B_\delta(f(x))$ is arbitrary, we obtain
\[
B_\delta(f(x))\subset f(B_{L\delta}(x)),
\]
thus the implication $(1)\implies(2)$ has been proved.

Next, we prove the implication $(2)\implies(1)$. The proof is similar to that of Lemma 1.1 but we use the assumption that $d$ is an ultrametric.  Fix $\delta_0>0$ as in (2). For any $x,y,z\in X$ with $d(f(x),y)\le\delta$ and $d(y,z)\le L\delta$, since
\[
d(f(x),z)\le\max\{d(f(x),y),d(y,z)\}\le\delta,
\]
we have
\[
z\in B_\delta(f(x))\subset f(B_{L\delta}(x)). 
\]
It follows that
\[
B_{L\delta}(y)\subset f(B_{L\delta}(x))
\]
for all $x,y\in X$ with $d(f(x),y)\le\delta$. Let $\xi=(x_i)_{i\ge0}$ be a $\delta$-pseudo orbit of $f$. For any $k>0$ and $0\le i\le k-1$, we have
\[
B_{L\delta}(x_{i+1})\subset f(B_{L\delta}(x_i))
\]
because $d(f(x_i),x_{i+1})\le\delta$. It follows that for any $p_k\in B_{L\delta}(x_k)$, there is a sequence $p_i\in B_{L\delta}(x_i)$, $0\le i\le k-1$, with $f(p_i)=p_{i+1}$ for all $0\le i\le k-1$, implying
\[
p_0\in\bigcap_{i=0}^k f^{-i}(B_{L\delta}(x_i)).
\]
We obtain
\[
\bigcap_{i=0}^\infty f^{-i}(B_{L\delta}(x_i))=\bigcap_{k>0}\bigcap_{i=0}^k f^{-i}(B_{L\delta}(x_i))\ne\emptyset;
\]
therefore, $\xi$ is $L\delta$-shadowed by some $q\in X$. Since $\xi$ is arbitrary, we conclude that $f$ satisfies the $L$-Lipschitz shadowing property, thus the implication $(2)\implies(1)$ has been proved. This completes the proof of the theorem. 
\end{proof}

\begin{rem}
\normalfont
By the above proof of the implication $(2)\implies(1)$, we see that if $d$ is an ultrametric on $X$ and if a continuous map $f\colon X\to X$ satisfies condition $(2)$ for $L=1$, then $f$ has the h-shadowing property and the $1$-Lipschitz shadowing property. 
\end{rem}

\section{Spectral decomposition: proofs of Theorems 1.7, 1.8 and 1.9}

In this section, we prove Theorems 1.7, 1.8, and 1.9; and by Theorems 1.8 and 1.9, we give an alternative proof of Theorem 1.4. Let $f\colon X\to X$ be a continuous map. For any $x,y\in X$ and $\delta>0$, the notation $x\rightarrow_\delta y$ means that there is a $\delta$-chain $(x_i)_{i=0}^k$ of $f$ with $x_0=x$ and $x_k=y$. We write $x\rightarrow y$ if $x\rightarrow_\delta y$ for all $\delta>0$. Note that
\[
CR(f)=\{x\in X\colon x\rightarrow x\}.
\]
We define a relation $\leftrightarrow$ in
\[
CR(f)^2=CR(f)\times CR(f)
\]
by: for any $x,y\in CR(f)$, $x\leftrightarrow y$ if and only if $x\rightarrow y$ and $y\rightarrow x$. It follows that $\leftrightarrow$ is a closed equivalence relation in $CR(f)^2$ and satisfies $x\leftrightarrow f(x)$ for all $x\in CR(f)$. An equivalence class $C$ of $\leftrightarrow$ is called a {\em chain component} for $f$. We denote by $\mathcal{C}(f)$ the set of chain components for $f$.

A subset $S$ of $X$ is said to be $f$-invariant if $f(S)\subset S$. For an $f$-invariant subset $S$ of $X$, we say that $f|_S\colon S\to S$ is {\em chain transitive} if for any $x,y\in S$ and $\delta>0$, there is a $\delta$-chain $(x_i)_{i=0}^k$ of $f|_S$ with $x_0=x$ and $x_k=y$.

\begin{rem}
\normalfont
The following properties hold
\begin{itemize}
\item $CR(f)=\bigsqcup_{C\in\mathcal{C}(f)}C$,
\item every $C\in\mathcal{C}(f)$ is a closed $f$-invariant subset of $CR(f)$,
\item $f|_C\colon C\to C$ is chain transitive for all $C\in\mathcal{C}(f)$,
\item for any $f$-invariant subset $S$ of $X$, if $f|_S\colon S\to S$ is chain transitive, then $S\subset C$ for some $C\in\mathcal{C}(f)$.
\end{itemize}
\end{rem}

Given a continuous map $f\colon X\to X$ and $\delta>0$, we define an equivalence relation $\leftrightarrow_\delta$ in
\[
CR(f)^2=CR(f)\times CR(f)
\]
by: for any $x,y\in CR(f)$, $x\leftrightarrow_\delta y$ if and only if $x\rightarrow_\delta y$ and $y\rightarrow_\delta x$. Since $x\leftrightarrow_\delta f(x)$ for every $x\in CR(f)$, each equivalence class of $\leftrightarrow_\delta$ is an $f$-invariant subset of $CR(f)$. For any $x,y\in CR(f)$ with $d(x,y)\le\delta$, take $z\in CR(f)$ such that $f(z)=x$.
Then, for every $\delta>0$, there is a $\delta$-chain $(x_i)_{i=0}^k$ of $f$ such that $x_0=x$ and $x_k=z$. Since
\[
(x_0,x_1,\dots,x_k,y)
\]
is a $\delta$-chain of $f$, we obtain $x\rightarrow_\delta y$. Similarly, we obtain $y\rightarrow_\delta x$, thus any $x,y\in CR(f)$ with $d(x,y)\le\delta$ satisfies $x\leftrightarrow_\delta y$. It follows that each equivalence class of $\leftrightarrow_\delta$ is an open and closed subset of $CR(f)$. We denote by $\mathcal{C}_\delta(f)$ the (finite) set of equivalence classes of $\leftrightarrow_\delta$. Note that we have $\leftrightarrow_{\delta_1}\subset\leftrightarrow_{\delta_2}$ for all $0<\delta_1<\delta_2$ and
\[
\leftrightarrow=\bigcap_{\delta>0}\leftrightarrow_{\delta}.
\]

Here, we prove Theorems 1.7 and 1.8. The proof of Theorem 1.7 is similar to that of Lemma 1 of \cite{M} but we provide it for completeness. For a continuous map $f\colon X\to X$, a non-empty closed $f$-invariant subset $M$ of $X$ is said to be a {\em minimal set} for $f$ if closed $f$-invariant subsets of $M$ are only $\emptyset$ and $M$. This condition is equivalent to $M=\omega(x,f)$ for all $x\in M$. We say that $x\in X$ is a {\em minimal point} for $f$ if $x\in\omega(x,f)$ and $\omega(x,f)$ is a minimal set for $f$. Let $M(f)$ denote the set of minimal points for $f$. By Zorn's lemma, we have $M(f)\ne\emptyset$. 

\begin{proof}[Proof of Theorem 1.7]
By assumption, we have $\delta_0>0$ such that for any $0<\delta\le\delta_0$, every $\delta$-pseudo orbit of $f$ is $L\delta$-shadowed by some point of $X$. Let $0<\delta\le\delta_0$ and let $\xi=(x_i)_{i\ge0}$ be a $\delta$-pseudo orbit of $f|_{CR(f)}$. Then, there is $C\in\mathcal{C}_\delta(f)$ with $x_i\in C$ for all $i\ge0$. For any $k>0$, since $x_0,x_k\in C$, we have a $\delta$-chain $(y_i)_{i=0}^l$ of $f$ with $y_0=x_k$ and $y_l=x_0$. Since\[
\xi_k=(x_0,x_1,\dots,x_{k-1},y_0,y_1,\dots,y_{l-1},x_0,x_1,\dots,x_{k-1},y_0,y_1,\dots,y_{l-1},\dots)
\]
is a $\delta$-pseudo orbit of $f$, letting
\[
X_k=\{x\in X\colon\text{$\xi_k$ is $L\delta$-shadowed by $x$}\},
\]
we have $X_k\ne\emptyset$. Note that $X_k$ is a closed $f^{k+l}$-invariant subset of $X$ and so satisfies $X_k\cap M(f^{k+l})\ne\emptyset$. By $M(f^{k+l})\subset CR(f)$, we obtain $X_k\cap CR(f)\ne\emptyset$ and so
\[
CR(f)\cap\bigcap_{i=0}^k f^{-i}(B_{L\delta}(x_i))\ne\emptyset.
\]
It follows that
\[
CR(f)\cap\bigcap_{i=0}^\infty f^{-i}(B_{L\delta}(x_i))=\bigcap_{k>0}\left(CR(f)\cap\bigcap_{i=0}^k f^{-i}(B_{L\delta}(x_i))\right)\ne\emptyset,
\]
that is, $\xi$ is $L\delta$-shadowed by some $x\in CR(f)$. Since $\xi$ is arbitrary, we conclude that
\[
f|_{CR(f)}\colon CR(f)\to CR(f)
\]
satisfies the $L$-Lipschitz shadowing property, completing the proof of the theorem.
\end{proof}

\begin{proof}[Proof of Theorem 1.8]
By assumption and Theorem 1.7, we have $\delta_0>0$ such that for any $0<\delta\le\delta_0$, every $\delta$-pseudo orbit of $f|_{CR(f)}$ is $L\delta$-shadowed by some point of $CR(f)$. In order to obtain a contradiction, we assume that $\mathcal{C}(f)$ is an infinite set. Then, there is $\delta>0$ such that, letting
\[
\Delta=\min\{d(C,C')\colon\text{$C,C'\in\mathcal{C}_\delta(f)$ and $C\ne C'$}\},
\]
we have $0<\Delta\le\delta_0$. We fix $C_0,C'_0\in\mathcal{C}_\delta(f)$ with $\Delta=d(C_0,C'_0)$ and $p\in C_0$, $q\in C'_0$ with $d(p,q)=d(C_0,C'_0)$. By taking $r\in C_0$ with $f(r)=p$, because $(r,q)$ is a $\Delta$-chain of $f|_{CR(f)}$, we have
\[
\max\{d(x,r),d(f(x),q)\}\le L\Delta<\Delta
\]
for some $x\in CR(f)$. By definition of $\Delta$, we obtain $x\in C_0$ and $f(x)\in C'_0$, which contradicts $f(x)\in C_0$ and $C_0\ne C'_0$. It follows that $\mathcal{C}(f)$ is a finite set. Then, since every $C\in\mathcal{C}(f)$ is an open and closed $f$-invariant subset of $CR(f)$, we easily see that
\[
f|_C\colon C\to C
\]
satisfies the $L$-Lipschitz shadowing property, proving the theorem.
\end{proof}

Let $f\colon X\to X$ be a chain transitive map. Given $\delta>0$, the {\em length} of a $\delta$-cycle $(x_i)_{i=0}^k$ of $f$ is defined to be $k$. Let $m=m(\delta)>0$ be the greatest common divisor of the lengths of all $\delta$-cycles of $f$. A relation $\sim_\delta$ in
\[
X^2=X\times X
\]
is defined by: for any $x,y\in X$, $x\sim_\delta y$ if and only if there is a $\delta$-chain $(x_i)_{i=0}^k$ of $f$ with $x_0=x$, $x_k=y$, and $m|k$. 

\begin{rem}
\normalfont
The following properties hold 
\begin{itemize}
\item[(P1)] $\sim_\delta$ is an open and closed $(f\times f)$-invariant equivalence relation in $X^2$,
\item[(P2)] any $x,y\in X$ with $d(x,y)\le\delta$ satisfies $x\sim_\delta y$, so for every $\delta$-chain $(x_i)_{i=0}^k$ of $f$, we have $f(x_i)\sim_\delta x_{i+1}$ for each $0\le i\le k-1$, implying $x_i\sim_\delta f^i(x_0)$ for every $0\le i\le k$,
\item[(P3)] for any $x\in X$ and $n\ge0$, $x\sim_\delta f^{mn}(x)$,
\item[(P4)] there exists $N>0$ such that for any $x,y\in X$ with $x\sim_\delta y$ and $n\ge N$, there is a $\delta$-chain $(x_i)_{i=0}^k$ of $f$ with $x_0=x$, $x_k=y$, and $k=mn$.
\end{itemize}
\end{rem}

Fix $x\in X$ and let $D_i$, $i\ge0$, denote the equivalence class of $\sim_\delta$ including $f^i(x)$. Then, $D_m=D_0$, and
\[
X=\bigsqcup_{i=0}^{m-1}D_i
\]
gives the partition of $X$ into the equivalence classes of $\sim_\delta$. Note that every $D_i$, $0\le i\le m-1$, is an open and closed subset of $X$ and satisfies $g(D_i)=D_{i+1}$. We call
\[
\mathcal{D}_\delta(f)=\{D_i\colon 0\le i\le m-1\}
\]
the {\em $\delta$-cyclic decomposition} of $X$.

\begin{defi}
\normalfont
We define a relation $\sim$ in
\[
X^2=X\times X
\]
by: for any $x,y\in X$, $x\sim y$ if and only if $x\sim_\delta y$ for every $\delta>0$, which is a closed $(f\times f)$-invariant equivalence relation in $X^2$. We denote by $\mathcal{D}(f)$ the set of equivalence classes of $\sim$.
\end{defi}

\begin{rem}
\normalfont
The relation $\sim$ was introduced in \cite{S} and rediscovered in \cite{RW} (based on the argument given in \cite[Exercise 8.22]{A}). We say that $(x,y)\in X^2$ is a {\em chain proximal pair} for $f$ if for any $\delta>0$, there is a pair of $\delta$-chains
\[
((x_i)_{i=0}^k,(y_i)_{i=0}^k)
\]
of $f$ with $(x_0,y_0)=(x,y)$ and $x_k=y_k$. As stated in Remark 8 of \cite{RW}, it holds that for any $x,y\in X$, $x\sim y$ if and only if $(x,y)$ is a chain proximal pair for $f$. By this equivalence, we say that $\sim$ is a {\em chain proximal relation} for $f$.
\end{rem}

\begin{rem}
\normalfont
\begin{enumerate}
\item Given any $0<\delta_1<\delta_2$, $x\sim_{\delta_1} y$ implies $x\sim_{\delta_2} y$ for all $x,y\in X$, so we have $\sim_{\delta_1}\subset\sim_{\delta_2}$. Note that
\[
\sim=\bigcap_{\delta>0}\sim_{\delta}.
\]
\item We say that a continuous map $f\colon X\to X$ is {\em chain mixing} if for any $x,y\in X$ and $\delta>0$, there exists $N>0$ such that for each $k\ge N$, there is a $\delta$-chain $(x_i)_{i=0}^k$ of $f$ with $x_0=x$ and $x_k=y$. We easily see that if $f$ is mixing, then $f$ is chain mixing, and the converse holds if $f$ has the shadowing property. If $|\mathcal{D}(f)|<\infty$, then, letting $m=|\mathcal{D}(f)|$, we see that
\[
f^m|_D\colon D\to D
\]
is chain mixing for all $D\in\mathcal{D}(f)$.
\end{enumerate}
\end{rem}

We prove Theorem 1.9. The proof is similar to that of Theorem 1.8.

\begin{proof}[Proof of Theorem 1.9]
By assumption, we have $\delta_0>0$ such that for any $0<\delta\le\delta_0$, every $\delta$-pseudo orbit of $f$ is $L\delta$-shadowed by some point of $X$. In order to obtain a contradiction, we assume that $\mathcal{D}(f)$ is an infinite set. Then, there is $\delta>0$ such that, letting
\[
\Delta=\min\{d(D,D')\colon\text{$D,D'\in\mathcal{D}_\delta(f)$ and $D\ne D'$}\},
\]
we have $0<\Delta\le\delta_0$. We fix $D_0,D'_0\in\mathcal{D}_\delta(f)$ with $\Delta=d(D_0,D'_0)$ and $p\in D_0$, $q\in D'_0$ with $d(p,q)=d(D_0,D'_0)$. By taking $D''_0\in\mathcal{D}_\delta(f)$ with $f(D''_0)=D_0$ and $r\in D''_0$ with $f(r)=p$, because $(r,q)$ is a $\Delta$-chain of $f$, we have
\[
\max\{d(x,r),d(f(x),q)\}\le L\Delta<\Delta
\]
for some $x\in X$. By definition of $\Delta$, we obtain $x\in D''_0$ and $f(x)\in D'_0$, which contradicts $f(x)\in D_0$ and $D_0\ne D'_0$. It follows that $\mathcal{D}(f)$ is a finite set. Since $f$ has the $L$-Lipschitz shadowing property, letting $m=|\mathcal{D}(f)|$, we see that
\[
f^m\colon X\to X
\]
also satisfies the $L$-Lipschitz shadowing property. Then, since every $D\in\mathcal{D}(f)$ is an open and closed $f^m$-invariant subset of $X$, we easily see that
\[
f^m|_D\colon D\to D
\]
satisfies the $L$-Lipschitz shadowing property, proving the theorem.
\end{proof}

Finally, by Theorems 1.8 and 1.9, we give an alternative proof of Theorem 1.4.

\begin{proof}[An alternative proof of Theorem 1.4]
Since $f$ has the contractive shadowing property, there is $0<L<1$ such that $f$ has the $L$-Lipschitz shadowing property. By Theorem 1.8, $\mathcal{C}(f)$ is a finite set. As in the proof of Theorem 1.4 in Section 2, we have $X=R(f)$ and so
\[
X=CR(f)=\bigsqcup_{C\in\mathcal{C}(f)}C.
\]
It remains to prove that every $C\in\mathcal{C}(f)$ is a finite set. Let $C\in\mathcal{C}(f)$. By Theorem 1.8,
\[
f|_{C}\colon C\to C
\]
has the $L$-Lipschitz shadowing property. From Theorem 1.8, it follows that $|\mathcal{D}(f|_C)|<\infty$ and for every $D\in\mathcal{D}(f|_C)$,
\[
f^m|_{D}:D\to D,
\]
where $m=|\mathcal{D}(f|_C)|$, is chain mixing and satisfies the $L$-Lipschitz shadowing property. Let $D\in\mathcal{D}(f|_C)$. Then, by Theorem 1.1, $f^m|_{D}\colon D\to D$ is a mixing homeomorphism with the h-shadowing property. As mentioned in Remark 1.2 (2), it follows that $f^m|_{D}\colon D\to D$ is locally eventually onto. Since $f^m|_{D}\colon D\to D$ is a homeomorphism, $D$ is a singleton. Note that
\[
C=\bigsqcup_{D\in\mathcal{D}(f|_C)}D.
\]
Since $D\in\mathcal{D}(f|_C)$ is arbitrary, we conclude that
\[
|C|=|\mathcal{D}(f|_C)|=m,
\]
proving the claim. Thus, Theorem 1.4 has been proved. 
\end{proof}

\section{Examples}

In this final section, we give several examples.

\begin{ex}
\normalfont
Let $X=\{0,1\}^\mathbb{N}$. For $\alpha>1$, we define a metric $d$ on $X$ by
\[
d(x,y)=\sup_{n\ge1}\alpha^{-n}|x_n-y_n|
\]
for all $x=(x_n)_{n\ge1}, y=(y_n)_{n\ge1}\in X$. Let $\sigma\colon X\to X$ be the shift map, that is,
\[
\sigma(x)_n=x_{n+1}
\]
for all $x=(x_n)_{n\ge1}\in X$ and $n\ge1$. Given any $\delta>0$, we shall show that every $\delta$-pseudo orbit of $\sigma$ is $\alpha^{-1}\delta$-shadowed by some point of $X$. Let $\xi=(x^{(i)})_{i\ge0}$ be a $\delta$-pseudo orbit of $\sigma$. Let $x=(x_1^{(n-1)})_{n\ge1}$ and note that $x\in X$. If $\delta\ge\alpha^{-1}$, then for every $i\ge0$, since
\[
\sigma^i(x)_1=x_{i+1}=x_1^{(i)},
\]
we have
\[
d(\sigma^i(x),x^{(i)})\le\alpha^{-2}\le\alpha^{-1}\delta,
\]
that is, $\xi$ is $\alpha^{-1}\delta$-shadowed by $x$. We assume $0<\delta<\alpha^{-1}$ and take $N\ge1$ such that
\[
\alpha^{-N-1}\le\delta<\alpha^{-N}.
\]
For any $i\ge0$, since
\[
d(\sigma(x^{(i)}),x^{(i+1)})\le\delta<\alpha^{-N},
\]
we have
\[
x_{n+1}^{(i)}=\sigma(x^{(i)})_n=x_n^{(i+1)}
\]
for all $1\le n\le N$. It follows that for every $i\ge0$,
\[
\sigma^i(x)_n=x_{i+n}=x_1^{(i+n-1)}=x_2^{(i+n-2)}=\cdots=x_n^{(i)}
\]
for all $1\le n\le N+1$, thus we obtain
\[
d(\sigma^i(x),x^{(i)})\le\alpha^{-N-2}\le\alpha^{-1}\delta
\]
for all $i\ge0$, that is, $\xi$ is $\alpha^{-1}\delta$-shadowed by $x$. Since $\xi$ is arbitrary, we conclude that $\sigma$ has the $\alpha^{-1}$-Lipschitz shadowing property. Note that $\sigma$ satisfies
\[
d(\sigma(x),\sigma(y))\le\alpha d(x,y)
\]
for all $x,y\in X$. This shows that the upper bound $M^{-1}$ in Theorem 1.3 is sharp.
\end{ex}

\begin{ex}
\normalfont
Let $X=\{0,1\}^\mathbb{Z}$. For $\alpha>1$, we define a metric $d$ on $X$ by
\[
d(x,y)=\sup_{n\in\mathbb{Z}}\alpha^{-|n|}|x_n-y_n|
\]
for all $x=(x_n)_{n\in\mathbb{Z}}, y=(y_n)_{n\in\mathbb{Z}}\in X$. Let $\sigma\colon X\to X$ be the shift map. Given any $\delta>0$, we shall show that every $\delta$-pseudo orbit of $\sigma$ is $\delta$-shadowed by some point of $X$. Let $\xi=(x^{(i)})_{i\ge0}$ be a $\delta$-pseudo orbit of $\sigma$ and let
\[
x^{(i)}=\sigma^i(x^{(0)})
\]
for all $i\le0$. Let $x=(x_0^{(n)})_{n\in\mathbb{Z}}$ and note that $x\in X$. If $\delta\ge\alpha^{-1}$, then for every $i\ge0$, since
\[
\sigma^i(x)_0=x_i=x_0^{(i)},
\]
we have
\[
d(\sigma^i(x),x^{(i)})\le\alpha^{-1}\le\delta,
\]
that is, $\xi$ is $\delta$-shadowed by $x$. We assume $0<\delta<\alpha^{-1}$ and take $N\ge1$ such that
\[
\alpha^{-N-1}\le\delta<\alpha^{-N}.
\]
For any $i\in\mathbb{Z}$, since
\[
d(\sigma(x^{(i)}),x^{(i+1)})\le\delta<\alpha^{-N},
\]
we have
\[
x_{n+1}^{(i)}=\sigma(x^{(i)})_n=x_n^{(i+1)}
\]
for all $-N\le n\le N$. If $-N\le n\le-1$, we have
\[
\sigma^i(x)_n=x_{i+n}=x_0^{(i+n)}=x_{-1}^{(i+n+1)}=\cdots=x_n^{(i)}.
\]
When $n=0$, we have
\[
\sigma^i(x)_0=x_i=x_0^{(i)}.
\]
If $1\le n\le N+1$, we have
\[
\sigma^i(x)_n=x_{i+n}=x_0^{(i+n)}=x_1^{(i+n-1)}=\cdots=x_n^{(i)}.
\]
If follows that
\[
d(\sigma^i(x),x^{(i)})\le\alpha^{-N-1}\le\delta
\]
for all $i\ge0$, that is, $\xi$ is $\delta$-shadowed by $x$. Since $\xi$ is arbitrary, we conclude that $\sigma$ has the $1$-Lipschitz shadowing property.
\end{ex}

\begin{ex}
\normalfont
Let $X$ be a compact metric space and let $Y=X^\mathbb{N}$. For $\alpha>1$, we define a metric $D$ on $Y$ by
\[
D(x,y)=\sup_{n\ge1}\alpha^{-n}d(x_n,y_n)
\]
for all $x=(x_n)_{n\ge1}, y=(y_n)_{n\ge1}\in Y$. Let $\sigma\colon Y\to Y$ be the shift map, that is,
\[
\sigma(x)_n=x_{n+1}
\]
for all $x=(x_n)_{n\ge1}\in Y$ and $n\ge1$. Given any $\delta>0$, we shall show that every $\delta$-pseudo orbit of $\sigma$ is $\delta/(\alpha-1)$-shadowed by some point of $Y$. Let $\xi=(x^{(i)})_{i\ge0}$ be a $\delta$-pseudo orbit of $\sigma$. Let $x=(x_1^{(n-1)})_{n\ge1}$ and note that $x\in Y$. Since
\[
D(\sigma(x^{(i)}),x^{(i+1)})\le\delta
\]
for all $i\ge0$, it holds that for every $i\ge0$,
\begin{align*}
d(\sigma^i(x)_n,x_n^{(i)})&=d(x_{i+n},x_n^{(i)})=d(x_1^{(i+n-1)},x_n^{(i)})\\&\le\sum_{j=1}^{n-1}d(x_j^{(i+n-j)},x_{j+1}^{(i+n-j-1)})=\sum_{j=1}^{n-1}d(x_j^{(i+n-j)},\sigma(x^{(i+n-j-1)})_j)\le\sum_{j=1}^{n-1}\alpha^j\delta
\end{align*}
for all $n\ge1$, thus we obtain
\[
D(\sigma^i(x),x^{(i)})\le\sup_{n\ge1}\left(\alpha^{-n}\cdot\sum_{j=1}^{n-1}\alpha^j\delta\right)\le\sum_{k=1}^\infty\alpha^{-k}\delta=\frac{1}{\alpha-1}\delta
\]
for all $i\ge0$, that is, $\xi$ is $\delta/(\alpha-1)$-shadowed by $x$. Since $\xi$ is arbitrary, we conclude that $\sigma$ has the $1/(\alpha-1)$-Lipschitz shadowing property. If $\alpha>2$, then $1/(\alpha-1)<1$. Note that $Y$ is infinite dimensional when, for example, $X=[0,1]$.
\end{ex}

\begin{ex}
\normalfont
Let $X$ be a compact metric space and let $Y=X^\mathbb{Z}$. For $\alpha>1$, we define a metric $D$ on $Y$ by
\[
D(x,y)=\sup_{n\in\mathbb{Z}}\alpha^{-|n|}d(x_n,y_n)
\]
for all $x=(x_n)_{n\in\mathbb{Z}}, y=(y_n)_{n\in\mathbb{Z}}\in Y$. Let $\sigma\colon Y\to Y$ be the shift map. Given any $\delta>0$, we shall show that every $\delta$-pseudo orbit of $\sigma$ is $\alpha\delta/(\alpha-1)$-shadowed by some point of $Y$. Let $\xi=(x^{(i)})_{i\ge0}$ be a $\delta$-pseudo orbit of $\sigma$ and let
\[
x^{(i)}=\sigma^i(x^{(0)})
\]
for all $i\le0$. Let $x=(x_0^{(n)})_{n\in\mathbb{Z}}$ and note that $x\in Y$. We have
\[
D(\sigma(x^{(i)}),x^{(i+1)})\le\delta
\]
for all $i\in\mathbb{Z}$. If $n\le-1$, we have
\begin{align*}
d(\sigma^i(x)_n,x_n^{(i)})&=d(x_{i+n},x_n^{(i)})=d(x_0^{(i+n)},x_n^{(i)})\le\sum_{j=0}^{-n-1}d(x_{-j}^{(i+n+j)},x_{-j-1}^{(i+n+j+1)})\\
&=\sum_{j=0}^{-n-1}d(\sigma(x^{(i+n+j)})_{-j-1},x^{(i+n+j+1)}_{-j-1})\le\sum_{j=0}^{-n-1}\alpha^{j+1}\delta,
\end{align*}
implying
\[
\alpha^n d(\sigma^i(x)_n,x_n^{(i)})\le\alpha^n\cdot\sum_{j=0}^{-n-1}\alpha^{j+1}\delta\le\sum_{k=0}^\infty\alpha^{-k}\delta=\frac{\alpha}{\alpha-1}\delta.
\]
When $n=0$, we have
\[
d(\sigma^i(x)_0,x_0^{(i)})=d(x_i,x_0^{(i)})=d(x_0^{(i)},x_0^{(i)})=0.
\]
If $n\ge1$, we have
\begin{align*}
d(\sigma^i(x)_n,x_n^{(i)})&=d(x_{i+n},x_n^{(i)})=d(x_0^{(i+n)},x_n^{(i)})\le\sum_{j=0}^{n-1}d(x_{j}^{(i+n-j)},x_{j+1}^{(i+n-j-1)})\\&=\sum_{j=0}^{n-1}d(x^{(i+n-j)}_j,\sigma(x^{(i+n-j-1)})_j)\le\sum_{j=0}^{n-1}\alpha^j\delta,
\end{align*}
implying
\[
\alpha^{-n}d(\sigma^i(x)_n,x_n^{(i)})\le\alpha^{-n}\cdot\sum_{j=0}^{n-1}\alpha^j\delta\le\sum_{k=1}^\infty\alpha^{-k}\delta=\frac{1}{\alpha-1}\delta.
\]
It follows that
\[
D(\sigma^i(x),x^{(i)})\le\frac{\alpha}{\alpha-1}\delta
\]
for all $i\ge0$, that is, $\xi$ is $\alpha\delta/(\alpha-1)$-shadowed by $x$. Since $\xi$ is arbitrary, we conclude that $\sigma$ has the $\alpha/(\alpha-1)$-Lipschitz shadowing property.
\end{ex}

\begin{ex}
\normalfont
Let $\mathbb{Z}_2=\{0,1\}$ and let $X=\{0,1\}^\mathbb{N}$. For $\alpha>1$, let $d$ be a metric on $X$ defined by
\[
d(x,y)=\sup_{n\ge1}\alpha^{-n}|x_n-y_n|
\]
for all $x=(x_n)_{n\ge1}, y=(y_n)_{n\ge1}\in X$. Note that $d$ is an ultrametric on $X$. We define a continuous surjection $f\colon X\to X$ by for any $x,y\in X$, $y=f(x)$ if and only if 
\[
y_n=x_n+x_{n+1}
\]
for all $n\ge1$. Given any $\delta>0$, we shall show that
\[
B_\delta(f(x))=f(B_{\alpha^{-1}\delta}(x))
\]
for all $x\in X$. If $\delta\ge1$, then
\[
B_\delta(f(x))=f(B_{\alpha^{-1}\delta}(x))=X
\]
for all $x\in X$. We assume $0<\delta<1$ and take $N\ge0$ such that
\[
\alpha^{-N-1}\le\delta<\alpha^{-N}.
\]
Note that
\[
B_{\alpha^{-1}\delta}(x)=\{y=(y_n)_{n\ge1}\in X\colon\text{$x_n=y_n$ for all $1\le n\le N+1$}\}
\]
for all $x\in X$. If $N=0$, we have
\[
B_\delta(f(x))=f(B_{\alpha^{-1}\delta}(x))=X
\]
for all $x\in X$. If $N\ge1$, we have
\[
B_\delta(f(x))=f(B_{\alpha^{-1}\delta}(x))=\{y=(y_n)_{n\ge1}\in X\colon\text{$f(x)_n=y_n$ for all $1\le n\le N$}\}
\]
for all $x\in X$. In both cases, we have
\[
B_\delta(f(x))=f(B_{\alpha^{-1}\delta}(x))
\]
for all $x\in X$. By Theorem 1.6, we conclude that $f$ satisfies the $\alpha^{-1}$-Lipschitz shadowing property.
\end{ex}

\begin{ex}
\normalfont
Let $X$ be a totally disconnected compact metric space endowed with an ultrametric $D$. For any equicontinuous homeomorphism $f\colon X\to X$, let $d$ be a metric on $X$ defined by
\[
d(x,y)=\sup_{i\in\mathbb{Z}}D(f^i(x),f^i(y))
\]
for all $x,y\in X$. It follows that $d$ is an ultrametric on $X$ equivalent to $D$ and $f$ is an isometry with respect to $d$, that is, $f$ satisfies
\[
d(f(x),f(y))=d(x,y)
\]
for all $x,y\in X$. This implies that
\[
B_{\delta}(f(x))=f(B_\delta(x))
\]
for all $\delta>0$ and $x\in X$, thus as stated in Remark 3.1, $f$ satisfies the h-shadowing property and the $1$-Lipschitz shadowing property (with respect to $d$). 
\end{ex}

\begin{ex}
\normalfont
Let $X=S^1=\{z\in\mathbb{C}\colon |z|=1\}$ and let $d$ be the arc-length metric on $X$.
For $n\ge2$, we define a continuous surjection $f_n\colon X\to X$ by
\[
f_n(z)=z^n
\]
for all $z\in X$. Then, there is $\delta_n>0$ such that if $0<\delta\le\delta_n$, then
\[
B_{n\delta}(f_n(z))\subset f_n(B_\delta(z))
\]
for all $z\in X$. For any $0<\Delta\le(n-1)\delta_n$, since
\[
0<\frac{1}{n-1}\Delta\le\delta_n,
\]
we have
\[
B_{\Delta+\frac{1}{n-1}\Delta}(f_n(z))=B_{n\cdot\frac{1}{n-1}\Delta}(f_n(z))\subset f_n(B_{\frac{1}{n-1}\Delta}(z))
\]
for all $z\in X$. By Lemma 1.1, we conclude that every $\Delta$-pseudo orbit of $f$ is $\Delta/(n-1)$-shadowed by some point of $X$. It follows that $f$ satisfies the $1/(n-1)$-Lipschitz shadowing property. Note that $1/(n-1)<1$ if $n\ge3$.
\end{ex}

\begin{ex}
\normalfont
Let $X=\{0,1\}^\mathbb{N}$ and let $d$ be a metric on $X$ defined by
\[
d(x,y)=\sup_{n\ge1}2^{-n}|x_n-y_n|
\]
for all $x=(x_n)_{n\ge1}, y=(y_n)_{n\ge1}\in X$.  For $K\ge2$, we define a continuous map $f_K\colon X\to X$ by
\[
f_K(x)_n=x_{Kn}
\]
for all $x=(x_n)_{n\ge1}\in X$ and $n\ge1$. Given any $\delta>0$, we shall show that every $\delta$-pseudo orbit of $f_K$ is $\delta^K$-shadowed by some point of $X$. Let $\xi=(x^{(i)})_{i\ge0}$ be a $\delta$-pseudo orbit of $f_K$. We define $x=(x_n)_{n\ge1}\in X$ by
\[
x_n=x_\alpha^{(\beta)}
\]
for all $n=\alpha K^{\beta}$ where $\alpha\ge1$ and $\beta\ge0$ are integers with $K{\not|}\alpha$. If $\delta\ge2^{-1}$, then for every $i\ge0$, since
\[
f_K^i(x)_n=x_{K^in}=x_n^{(i)}
\]
for all $1\le n\le K-1$, we have
\[
d(f_K^i(x),x^{(i)})\le2^{-K}\le\delta^K,
\]
that is, $\xi$ is $\delta^K$-shadowed by $x$. We assume $0<\delta<2^{-1}$ and take $N\ge1$ such that
\[
2^{-N-1}\le\delta<2^{-N}.
\]
For any $i\ge0$, since
\[
d(f_K(x^{(i)}),x^{(i+1)})\le\delta<2^{-N},
\]
we have
\[
x_{Kn}^{(i)}=f_K(x^{(i)})_n=x_n^{(i+1)}
\]
for all $1\le n\le N$. Given any
\[
1\le n\le KN+K-1,
\]
letting $n=\alpha K^\beta$ where $\alpha\ge1$ and $\beta\ge0$ are integers with $K{\not|}\alpha$, we have
\[
f_K^i(x)_n=x_{K^in}=x_{\alpha K^{i+\beta}}=x_\alpha^{(i+\beta)}
\]
for all $i\ge0$. When $\beta=0$, we obtain
\[
f_K^i(x)_n=x_\alpha^{(i)}=x_n^{(i)}
\]
for all $i\ge0$. If $\beta\ge1$, then since
\[
1\le n=\alpha K^\beta\le KN+K-1
\]
and so
\[
1\le \alpha K^{\beta-1}\le N,
\]
we obtain
\[
x_n^{(i)}=x_{\alpha K^\beta}^{(i)}=x_{\alpha K^{\beta-1}}^{(i+1)}=\cdots=x_\alpha^{(i+\beta)},
\]
implying 
\[
f_K^i(x)_n=x_\alpha^{(i+\beta)}=x_n^{(i)}
\]
for all $i\ge0$. It follows that for every $i\ge0$,
\[
f_K^i(x)_n=x_n^{(i)}
\]
for all $1\le n\le KN+K-1$, thus we obtain
\[
d(f_K^i(x),x^{(i)})\le2^{-KN-K}\le\delta^K,
\]
that is, $\xi$ is $\delta^K$-shadowed by $x$, completing the proof of the claim. As a consequence, for any $K\ge2$, $f_K$ satisfies the $L$-Lipschitz shadowing property for all $L>0$.
\end{ex}

\end{document}